\documentclass{article}
\usepackage{graphicx} 
\usepackage[utf8]{inputenc}
\usepackage{amsmath}
\usepackage{amsthm}
\usepackage{amssymb}
\newtheorem{theorem}{Theorem}[section]
\newtheorem{conjecture}{Conjecture}[section]
\newtheorem{corollary}{Corollary}[section]

\title{Properties of Prime Products and Robin's Inequality}
\author{William McCann}
\date{February 2020}

\begin{document}

\maketitle

\begin{abstract}

There are many formulations of problems that have been proven to be equivalent to the Riemann Hypothesis in modern mathematics. In this paper we look at the formulation of an inequality derived by Robin in 1984\cite{robin1984grandes} that proves the Riemann Hypothesis true if and only if Robin's Inequality is true for all $n\geq 5040$. In this paper we look at explicit products of prime numbers and show that if a given $n=p_1^{k_1}p_2^{k_2}\ldots p_m^{k_m}$ satisfies Robin's Inequality, then $n=P_1^{k_1}P_2^{k_2}\ldots P_m^{k_m}$ with $p_j \leq P_j$ also satisfies the inequality. We also then offer two conjectures that if proven individually could imply some interesting properties of the Riemann Hypothesis and Robin's Inequality.

\end{abstract}

\section{Introduction}

    The Riemann Hypothesis is one of the greatest unsolved problems still open in our time. In its classic formulation, the hypothesis claims that all
    non-trivial zeros of the analytic continuation of the Riemann Zeta function
    $$
    \zeta (s) = \sum_{n=1}^\infty \dfrac{1}{n^s}
    $$
    occur only on the critical line $s=\frac{1}{2}+ib$. 
    
    While this formulation is certainly the most famous (or infamous depending on which burned out 
    mathematician you ask) there have been many other equivalent problems discovered over the years. In this paper we will be looking at one of these
    equivalent formulations: Robin's Inequality. According to Robin's Theorem, the Riemann Hypothesis is true if and only if the following inequality 
    $$
    \sigma (n) < e^\gamma n\log (\log(n))
    $$
    holds for all $n>5040$\cite{robin1984grandes}, where $\gamma=.5772\ldots$ is the Euler-Mascheroni Constant \cite{weisstein2002euler} and $\sigma(n)$ is the sum of divisors function. 
    
    Since this formulation was first published by Robin in $1984$, several interesting developments have occurred, such as the discovery that the Riemann Hypothesis is true if and only if all numbers divisible by a $5$th power number satisfy Robin's Inequality\cite{JTNB_2007__19_2_357_0}. Most discoveries are similar to this in that they consider a given $n$ and its respective prime factors. However, in this paper we consider products of primes and come up with some small, fun results.
    
    Before looking at our results, we must first look at some formulas that we will be using. The first is a closed form product for the sum of divisors function $\sigma(n)$. Given that $n=p_1^{k_1}p_2^{k_2}\ldots p_m^{k_m}$ where $p_1,\ldots,p_m$ are prime numbers, we can write $\sigma(n)$ as 
    $$
    \prod_{j=1}^m\dfrac{p_j^{k_j+1}-1}{p_j-1}.\cite{berndt2012ramanujan}
    $$
    If we plug this back in to the inequality we find that
    $$
    \prod_{j=1}^m\dfrac{p_j^{k_j+1}-1}{p_j-1} < e^\gamma n\log (\log(n)) 
    $$
    which if we plug back in that $n=p_1^{k_1}p_2^{k_2}\ldots p_m^{k_m}$ we can say that
    $$
    \prod_{j=1}^m\dfrac{p_j^{k_j+1}-1}{p_j-1} < e^\gamma\prod_{j=1}^m p_j^{k_j}\log \left(\log\left(\prod_{j=1}^m p_j^{k_j}\right)\right).
    $$
    Performing a little be of algebra we can divide both sides by $n$ and show that Robin's inequality is equivilent to 
    $$
    \prod_{j=1}^m\dfrac{p_j^{k_j+1}-1}{p_j^{k_j}\left(p_j-1\right)} < e^\gamma\log \left(\log\left(\prod_{j=1}^m p_j^{k_j}\right)\right) \equiv e^\gamma\log \left(\sum_{j=1}^m k_j\log( p_j)\right).
    $$
    Now with these formulations we can begin to prove some of the theorems presented in this paper.
    
\section{Robin's Inequality and Prime Products}

We begin with a small toy example to get ourselves acquainted with some of the ideas we will be considering in this paper. Consider a number that is a power of a single prime $n=p^k > 5040$. In this case both sides of the inequality are simplified down significantly to make the following theorem fairly easy. 

\begin{theorem}
    If $n>5040$ is of the form $n=p^k$, where $p$ is prime, then Robin's Inequality holds
\end{theorem}

\begin{proof}
    Using our expression of Robin's Inequality we know
    $$
    \prod_{j=1}^m\dfrac{p_j^{k_j+1}-1}{p_j^{k_j}\left(p_j-1\right)} < e^\gamma\log \left(\log\left(\prod_{j=1}^m p_j^{k_j}\right)\right)
    $$
    which is equal to
    $$
        \dfrac{p^{k+1}-1}{p^k(p-1)} < e^\gamma \log(\log(p^k))
    $$
    since $n=p^k$. If we break up the fraction, we can find that
    $$
        \dfrac{p}{p-1} - \dfrac{1}{p^k(p-1)} < e^\gamma \log(\log(p^k)).
    $$
    Since $p\geq 2$, the quantity on the left hand side is less than $2$ for all possible primes, which makes our expression
    $$
        2 < e^\gamma \log(\log(p^k)).
    $$
    Our right hand side is a monotonically increasing function, so if we can show that this inequality holds true for $5040$, then it must hold true in general for $p^k > 5040$
    $$
        2 < e^\gamma \log\log(5040) \approx 3.817 < e^\gamma \log(\log(p^k))
    $$
    therefore our inequality holds for all primes of the given form.
\end{proof}

As noted earlier, this proof was fairly simple, and has been shown before in some capacity. However, the next result is one that we are interested in personally, as it opens up a large class of valid numbers for Robin's inequality.

\begin{theorem}
    If Robin's Inequality holds for $n = p_1^{k_1}p_2^{k_2}\dots p_m^{k_m}$ where $p_1,p_2,\ldots,p_m$ are $m$ distinct prime 
    numbers, then the inequality holds for any $n$ of the form $n = P_1^{k_1}P_2^{k_2}\dots P_m^{k_m}$ where $p_j \leq P_j$ 
\end{theorem}

\begin{proof}
    Notice that we only need to show that for a particular $n'=p_1^{k_1}\ldots P_j^{k_j}\ldots p_m^{k_m}$ is valid for $p_j \leq P_j$, as if that is valid, we can continually increase every other prime factor individually and have the resultant number still be valid. As such we show that for a single increased prime the inequality still holds.
    
    Suppose we have $n= p_1^{k_1}p_2^{k_2}\dots p_m^{k_m}$, then our formulation of Robin's inequality is as follows
    $$
        \prod_{j=1}^m\dfrac{p_j^{k_j+1}-1}{p_j^{k_j}\left(p_j-1\right)} < e^\gamma\log \left(\log\left(\prod_{j=1}^m p_j^{k_j}\right)\right).
    $$
    Now we will show that if this inequality holds, and we increase any specific $p_j$ the R.H.S. will increase, and the L.H.S. will decrease which implies the inequality still holds. 
    
    Trivially, since $\ln x$ is a monotonically increasing function, if we increase the value of any $p_j$, since $p_j > 1$ then we know that the value of $e^\gamma \log\log\left(p_1^{k_1}p_2^{k_2}\dots p_m^{k_m}\right)$ must also increase.
    
    Now we take the R.H.S. and differentiate it with respect to any arbitrary $p_s$, note the change in index labels to avoid confusion
    
    $$
        \dfrac{\partial}{\partial p_s}\prod_{j=0}^m \dfrac{p_j^{k_j+1}-1}{p_j^{k_j}(p_j-1)} = \dfrac{\partial}{\partial p_s}\left(\dfrac{p_s^{k_s+1}-1}{p_s^{k_s}(p_s-1)}\right)\cdot\prod_{\substack{j=0 \\ j\neq s}}^m \dfrac{p_j^{k_j+1}-1}{p_j^{k_j}(p_j-1)}.
    $$
    Notice that the value of 
    $$
    \prod_{\substack{j=0 \\ j\neq s}}^m \dfrac{p_j^{k_j+1}-1}{p_j^{k_j}(p_j-1)}
    $$
    is strictly positive. Therefore it is sufficient to show that 
    $$
        \dfrac{\partial}{\partial p_s}\left(\dfrac{p_s^{k_s+1}-1}{p_s^{k_s}(p_s-1)}\right) < 0
    $$
    for the entire derivative to be negative. Taking the derivative we find
    $$
        \dfrac{\partial}{\partial p_s}\left(\dfrac{p_s^{k_s+1}-1}{p_s^{k_s}(p_s-1)}\right) = \dfrac{k_s(p_s-1)-p_s(p_s^{k_s}-1)}{p_s^{k_s+1}(p_s-1)^2} < 0.
    $$
    Since $p_s > 1$, the bottom quantity is greater than $0$, so this implies
    $$
        k_s(p_s-1)-p_s(p_s^{k_s}-1) < 0 \equiv -p_s^{k_s+1} + (k+1)p_s - k < 0
    $$
    if we then take the derivative of this quantity we get that
    $$
        \dfrac{\partial}{\partial p_s}(-p_s^{k_s+1} + (k+1)p_s - k) = (k+1)(1-p_s^{k_s}).
    $$
    This second derivative of the numerator is negative for all values of $p_s > 1$. Since our derivative numerator is $0$ at $p_s = 1$, we can now say that our derivative is negative for all values of $p_s > 1$. 
    
    Therefore by increasing the value of any arbitrary $p_s$ will decrease the value of the L.H.S. of our inequality. Since we have decreased the L.H.S. and increased the R.H.S. of the inequality we assumed to be true, then the resultant inequality must also be true.
\end{proof}

Now lets apply this Theorem to a small toy application similar to the first theorem that we proved earlier. Again this is not super interesting, however it is a good insight into some potential thoughts of what could be done with this Theorem.

\begin{corollary}
    All numbers of the form $n=p_1^{k_1}p_2^{k_2} > 5040$ satisfy Robin's inequality
\end{corollary}
\begin{proof}
    Starting with our form of the inequality we have
    $$
        \dfrac{p_1^{k_1+1}-1}{p_1^{k_1}(p_1-1)}\dfrac{p_2^{k_2+1}-1}{p_2^{k_2+2}(p_2-1)} < e^\gamma \ln\ln\left(p_1^{k_1}p_2^{k_2}\right)
    $$
    and using the following inequalities we can show that the inequality holds regardless of values of $k_1,k_2$
    $$
        \dfrac{p_1^{k_1+1}-1}{p_1^{k_1}(p_1-1)}\dfrac{p_2^{k_2+1}-1}{p_2^{k_2+2}(p_2-1)} < \dfrac{p_1^{k_1+1}}{p_1^{k_1}(p_1-1)}\dfrac{p_2^{k_2+1}}{p_2^{k_2+2}(p_2-1)} = \dfrac{p_1p_2}{(p_1-1)(p_2-1)}.
    $$
    Since $n>5040$ we know that 
    $$
        e^\gamma \ln\ln\left(5040\right) < e^\gamma \log\log\left(p_1^{k_1}p_2^{k_2}\right).
    $$
    Now choose $p_1=2$ and $p_2=3$. We see that 
    $$
        \dfrac{2\cdot 3}{1\cdot 2} = 3 < e^\gamma \log\log\left(5040\right) \approx 3.817.
    $$
    Since the above inequality holds, we know that the inequality must hold for all values $n=2^{k_1}3^{k_2}$. 
    
    By \textbf{Theorem $2.2$}, since $2, 3$ are the first $2$ distinct primes, and the inequality holds for all $n=2^{k_1}3^{k_2}$, the inequality holds for all $n=p_1^{k_1}p_2^{k_2}$ where $p_1, p_2$ are distinct primes.
\end{proof}

And just like the above proof, we could do it again for $3$ distinct primes

\begin{corollary}
    All numbers of the form $n=p_1^{k_1}p_2^{k_2}p_3^{k_3}$ satisfy Robin's inequality
\end{corollary}
\begin{proof}
    The proof can be constructed similarly to the above, so that our inequality becomes
    $$
        \dfrac{2\cdot 3\cdot 5}{1\cdot 2 \cdot 4}\ = 3.75 < e^\gamma \log\log\left(5040\right) \approx 3.817.
    $$
    which is also true.
\end{proof}

Now lets see the boundary for which the bound of $e^gamma\log\log(5040)$ can be used. If we restrict ourselves to only first powers of primes, we can extend the product out as follows

\begin{corollary}
    All numbers of the form $n = p_1p_2p_3\dots p_m$ where $m \leq 9$ satisfies Robin's inequality.
\end{corollary}
\begin{proof}
    Starting as normal, we begin with Robin's inequality in our form after a little bit of algebra
    $$
        \prod_{k=1}^m \dfrac{p_k + 1}{p_k} < e^\gamma  \log\log\left(5040\right) \implies \prod_{k=1}^m \dfrac{p_k + 1}{p_k} < 3.817.
    $$
    
    By \textbf{Theorem $2.2$}, we only need to consider the first $m$ primes in order to be sure that this inequality holds for all $n$ of this form. This is satisfied by the first $9$ primes which have the value of 
    $$
        \dfrac{(2 + 1)}{2}\dfrac{(3 + 1)}{3}\dfrac{(5 + 1)}{5}\dfrac{(7 + 1)}{7}\dfrac{(11 + 1)}{11}\dfrac{(13 + 1)}{13}\dfrac{(17 + 1)}{17}\dfrac{(19 + 1)}{19}\dfrac{(23 + 1)}{23} \approx 3.748
    $$
    Since $3.748 < 3.817$, this holds for $9$ primes of the first power, and trivially for any $m < 9$ with only single powers as well.
\end{proof}

\section{Discussion and Further Work}

With Theorem $2.2$ we showed that given any prime number string $>5040$ that satisfies Robin's inequality, will still satisfy Robin's inequality if any of those prime numbers in increased. As a consequence of this, one only needs to consider the first $m$ distinct prime numbers when considering prime products of certain sizes, as all other products will hold true. Interestingly this leads us to the next question of whether or not all prime strings of single powers hold true for Robin's inequality. We conjecture the following: 

\begin{conjecture}
    For any $n>5040$ where $n$ is the product of the first $m$ distinct primes, Robin's Inequality holds true. 
\end{conjecture}

Computationally, we believe this to be true, as it effectively asks for the truth of the following statement
$$
\prod_{j=1}^{m}\dfrac{p_j+1}{p_j} < e^\gamma\log \left(\log\left(\prod_{j=1}^m p_j\right)\right).
$$
From the following figures $1$ and $2$ one can see that there appears to be a very clear difference between the two growth rates of the inequalities. If we denote the value of the $m^{th}$ partial product as $q_m$ and $\alpha_m$ for the left and right hand sides respectively, the ratio $\frac{\alpha_m}{q_m}$ appears to monotonically approach a limit of $\approx 1.63$, which would prove the conjecture.

\begin{figure}[h]
    \centering
    \begin{minipage}{0.45\textwidth}
        \centering
        \includegraphics[width=0.9\textwidth]{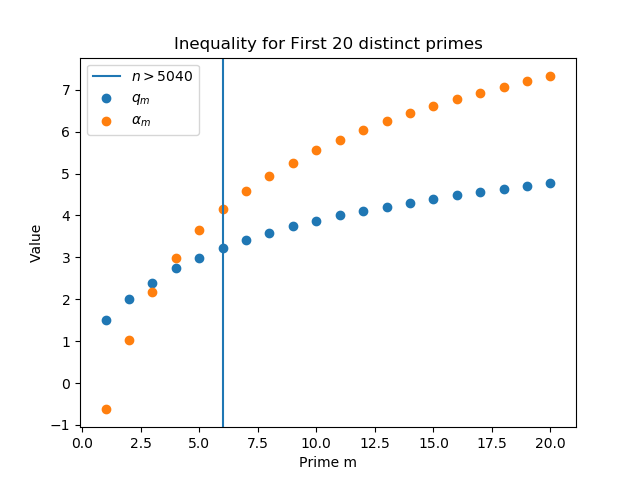} 
    \end{minipage}\hfill
    \begin{minipage}{0.45\textwidth}
        \centering
        \includegraphics[width=0.9\textwidth]{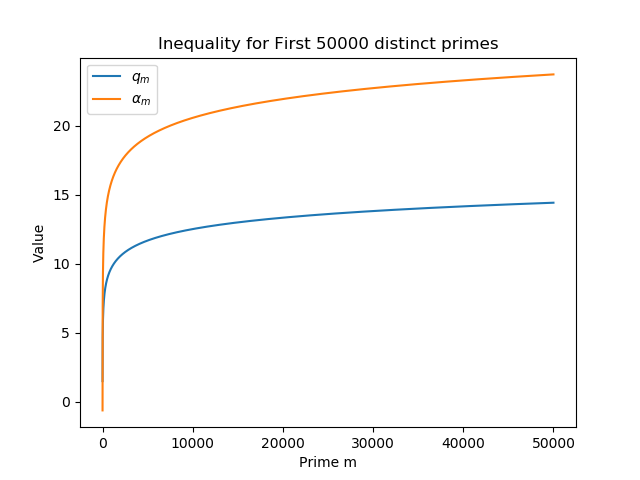} 
    \end{minipage}
    \caption{In these figures you can clearly see that the values of the right hand side of the inequality outpace the growth of the right hand side of the inequality in a way that seems to imply that the inequality holds. A horizontal line is drawn to distinguish where $n>5040$ for the first $m$ distinct primes.}
\end{figure}

While this does appear like it would be an easy limit difference to prove, the problem lies in the requirement of the distribution of primes, making a proof by induction seemingly impossible. However, if this was shown to be true we would only need one final piece to show that Robin's inequality holds for all possible numbers $n>5040$. Consider the following conjecture. 

\begin{conjecture}
    For any $n>5040$ where $n= p_1^{k_1}p_2^{k_2}\dots p_m^{k_m}$, then $n= p_1^{k_1}p_2^{k_2}\dots p_j^{k_j+1}\ldots p_m^{k_m}$ is also valid for all values of $j$
\end{conjecture}

If this and the prior conjecture were both true, it could be used in conjunction with Theorem $2.2$ to show that all possible prime products $>5040$ held true for the inequality as well. Base cases would need to be considered to avoid the potential problem with counterexamples when $n\leq5040$, however that would be expected to be trivial if one was to prove Conjectures $3.1$ and $3.2$. All counterexamples would also have to have between $4\leq m\leq 9$ distinct prime factors per the proofs of Corollaries $2.2$ and $2.3$, which means that only a small number of cases need be considered.

\nocite{*} 
\bibliographystyle{plain} 
\bibliography{Properties_of_Prime_Products_and_Robins_Inequality}

\begin{thebibliography}{1}

\bibitem{berndt2012ramanujan}
Bruce~C Berndt.
\newblock {\em Ramanujan’s notebooks}.
\newblock Springer Science \& Business Media, 2012.

\bibitem{JTNB_2007__19_2_357_0}
YoungJu Choie, Nicolas Lichiardopol, Pieter Moree, and Patrick Sol\'e.
\newblock On robin's criterion for the riemann hypothesis.
\newblock {\em Journal de th\'eorie des nombres de Bordeaux}, 19(2):357--372,
  2007.

\bibitem{robin1984grandes}
G~ROBIN.
\newblock Grandes valeurs de la fonction somme qes diviseurs et, hypoth{\`e}se
  deriemann.
\newblock 1984.

\bibitem{weisstein2002euler}
Eric~W Weisstein.
\newblock Euler-mascheroni constant.
\newblock 2002.

\end{thebibliography}

\end{document}